\documentclass[11pt]{amsart}
\usepackage{amssymb,amsmath, amsthm, amsfonts, tikz-cd}

\usepackage{graphicx}
\usepackage{listings}
\usepackage[margin=1in]{geometry}
\usepackage{lstautogobble}
\usepackage{enumerate}
\usepackage[shortlabels]{enumitem}
\usepackage{thmtools}
\usepackage{thm-restate}
\usepackage{amsthm}
\usepackage{verbatim}

\usepackage{mathtools}
\usepackage{physics}
\usepackage{color}
\usepackage{hyperref}
\usepackage[capitalise]{cleveref}
\usepackage[bottom]{footmisc}
\crefformat{equation}{(#2#1#3)}

\usepackage{float}
\restylefloat{table}

\usepackage{mathrsfs}
\setlist{  
  listparindent=\parindent,
  parsep=0pt,
}

\theoremstyle{plain}
\newtheorem{thm}{Theorem}[section]
\newtheorem{prop}[thm]{Proposition}
\newtheorem{lemma}[thm]{Lemma}

\theoremstyle{definition}

\newtheorem{remark}[thm]{Remark}

\Crefname{thm}{Theorem}{Theorems}
\Crefname{prop}{Proposition}{Propositions}

\numberwithin{equation}{section} 

\DeclarePairedDelimiter{\paren}{\lparen}{\rparen}

\newcommand{\M}{{\mathcal{M}}}

\newcommand{\p}{{\partial}}

\renewcommand{\d}{\delta}

\newcommand{\R}{{\mathbb{R}}}
\newcommand{\C}{{\mathbb{C}}}
\newcommand{\N}{{\mathbb{N}}}

\newcommand{\Z}{{\mathbb{Z}}}

\newcommand{\T}{{\mathbb{T}}}
\newcommand{\g}{{\mathsf{g}}}
\newcommand{\G}{{\mathsf{G}}}

\newcommand{\Sc}{{\mathcal{S}}}

\renewcommand{\M}{{\mathbb{M}}}
\newcommand{\I}{\mathbb{I}}

\newcommand{\ga}{\gamma}

\newcommand{\ph}{\phantom{=}}
\newcommand{\nn}{\nonumber}

\newcommand{\ep}{\epsilon}
\newcommand{\vep}{\varepsilon}
\newcommand{\al}{\alpha}
\newcommand{\be}{\beta}
\newcommand{\ka}{\kappa}
\newcommand{\la}{\lambda}

\newcommand{\Tc}{\mathcal{T}}

\newcommand{\indic}{\mathbf{1}}

\newcommand{\F}{{\mathcal{F}}}
\newcommand{\Dm}{|\nabla|}

\renewcommand{\P}{\mathcal{P}}

\let\div\relax
\DeclareMathOperator{\div}{\mathrm{div}}

\def\XXint#1#2#3{{\setbox0=\hbox{$#1{#2#3}{\int}$ }
\vcenter{\hbox{$#2#3$ }}\kern-.6\wd0}}

\let\oldtocsection=\tocsection
 
\let\oldtocsubsection=\tocsubsection
 
\let\oldtocsubsubsection=\tocsubsubsection
 
\renewcommand{\tocsection}[2]{\hspace{0em}\oldtocsection{#1}{#2}}
\renewcommand{\tocsubsection}[2]{\hspace{1em}\oldtocsubsection{#1}{#2}}
\renewcommand{\tocsubsubsection}[2]{\hspace{2em}\oldtocsubsubsection{#1}{#2}}

\title[Trend to equilibrium for flows with random diffusion]{Trend to equilibrium for flows with random diffusion}

\author[S. Aryan]{Shrey Aryan}
\email{shrey183@mit.edu}
\author[M. Rosenzweig]{Matthew Rosenzweig}
\email{mrosenzw@mit.edu}
\thanks{M.R. is supported in part by the NSF through grants DMS-2052651, DMS-2206085 and by the Simons Foundation through the Simons Collaboration on Wave Turbulence.}
\author[G. Staffilani]{Gigliola Staffilani}
\email{gigliola@math.mit.edu}
\thanks{G.S. is supported in part by the NSF through grant DMS-2052651 and by the Simons Foundation through the Simons Collaboration on Wave Turbulence.}

\begin{document}
\begin{abstract}
Motivated by the possibility of noise to cure equations of finite-time blowup, recent work \cite{RS2023} by the second and third named authors showed that with quantifiable high probability, \emph{random diffusion} restores global existence for a large class of active scalar equations in arbitrary dimension with possibly singular velocity fields. This class includes Hamiltonian flows, such as the SQG equation and its generalizations, and gradient flows, such as the Patlak-Keller-Segel equation. A question left open is the asymptotic behavior of the solutions, in particular, whether they converge to a steady state. We answer this question by showing that the solutions from \cite{RS2023} in the periodic setting converge in Gevrey norm exponentially fast to the uniform distribution as time $t\rightarrow\infty$. 
\end{abstract}
\maketitle

\section{Introduction}
Taking inspiration from \cite{GhV2014, BNSW2020}, recent work \cite{RS2023} by the second and third named authors showed for a large class of scalar flows that the addition of a \emph{random diffusion} to the dynamics leads to global classical solutions with high probability. Such an effect is significant, as without noise, the class considered includes equations, such as aggregation equations, for which finite-time blowup holds for classical solutions, as well equations such as the inviscid SQG equation, for which global existence of classical solutions is unknown. We refer to the introduction of \cite{RS2023} for a detailed discussion of the physical relevance and mathematical history of the class of equations considered.

A question left open in the cited work is the asymptotic behavior of solutions as $t\rightarrow\infty$. The purpose of this note is to answer this question by showing that with high probability, solutions converge to the uniform distribution with mass equal to that of the initial data. One may interpret this as ``equilibriation'' of the system. As the uniform distribution is a stationary solution, in particular, this implies that it is the unique equilibrium. The present work together with the previous works \cite{GhV2014, BNSW2020, RS2023}, demonstrate a fairly complete global theory for the effect of random damping/diffusion. 

\subsection{The model}\label{ssec:introM}
The stochastic partial differential equation (SPDE) we consider is
\begin{equation}\label{eq:spde}
\begin{cases}
\p_t\theta +\div(\theta\M\nabla\g\ast\theta)  = \nu\Dm^s\theta\dot{W} \\
\theta|_{t=0} = \theta^0
\end{cases}
\qquad (t,x)\in\R_+\times\T^d.
\end{equation}
Above, $\M$ is a $d\times d$ constant matrix with real entries and $\g\in\Sc'(\T^d)$ is a tempered distribution, such that there is a $\ga >0$ so that the Fourier transform $\hat{\g}(k)$ satisfies the bound
\begin{align}\label{eq:gassmp}
\forall k\in \Z^d, \qquad |\hat{\g}(k)| \lesssim |k|^{-\ga}.
\end{align}

The \emph{random diffusion} corresponds to the term in the right-hand side of \eqref{eq:spde}, where $\nu>0$, $\Dm^s$ is the fractional Laplacian of order $s$ (i.e., the Fourier multiplier with symbol $|k|^s$), and $W$ is a one-dimensional standard Brownian motion. The randomness stems from the fact that the diffusivity coefficient $\nu$ is modulated by the white noise $\dot{W}$. The addition of such a term was first proposed by Buckmaster et al. \cite{BNSW2020} to obtain global existence in the case of $\M, \g$ corresponding to the inviscid SQG equation, following an earlier random damping term proposed by Glatt-Holtz and Vicol \cite{GhV2014} in the case of $\M,\g$ corresponding to the $d=2$ incompressible Euler vorticity equation. In \cite{RS2023}, an inhomogeneous diffusion $\nu(1+\Dm^s)\theta\dot{W}$ was instead used because the problem was set on $\R^d$, which entails issues at low frequencies (see \cref{ssec:imtroCP} for further elaboration).

The mathematical interpretation of the SPDE \eqref{eq:spde} is based on a pathwise change of unknown. Supposing we have a solution $\theta$ to \eqref{eq:spde} and formally setting $\mu^t\coloneqq \Gamma^t\theta^t$, where for each realization of the Brownian motion $W$, $\Gamma^t\coloneqq e^{-\nu W^t\Dm^s}$ is the Fourier multiplier with symbol $e^{-\nu W^t|k|^s}$, It\^o's lemma implies
\begin{align}
\p_t\mu &= - \div\Gamma\paren*{\Gamma^{-1}\mu\M\nabla\g\ast\Gamma^{-1}\mu}- \frac{\nu^2}{2}\Dm^{2s}\mu. \label{eq:rpde}
\end{align}
See \cite[Section 2]{BNSW2020} or \cite[Equation (1.7)]{RS2023} for details of the computation and \cite[Remark 1.1]{RS2023} for an explanation of the choice of It\^o noise, as opposed to Stratonovich noise. Equation \eqref{eq:rpde} is a \emph{random PDE} that may be interpreted pathwise: for a fixed realization of $W$, which almost surely is a locally continuous path on $[0,\infty)$, one studies the Cauchy problem for \eqref{eq:rpde}. 

\subsection{Main results}\label{ssec:introMR}
To state our results, we first fix some notions. Here and throughout this paper, we assume that the potential $\g$ satisfies the condition \eqref{eq:gassmp}. We assume that we have a standard real Brownian motion $\{W^t\}_{t\geq 0}$ defined on a filtered probability space $(\Omega, \F, \{\F^t\}_{t\geq 0},\P)$ satisfying all the usual assumptions. For $\al,\be,\nu>0$, we define the event
\begin{equation}\label{eq:Omsetdef}
\Omega_{\al,\be,\nu} \coloneqq \{\omega\in\Omega : \al+\beta t - \nu W^t(\omega) \geq 0 \quad \forall t\in [0,\infty)\} \subset\Omega.
\end{equation}
It is known that $\P(\Omega_{\al,\be,\nu}) = 1 - e^{-\frac{2\al\be}{\nu^2}}$ \cite[Proposition 6.8.1]{Resnick1992}. The definition of the Fourier-Lebesgue space $\hat{W}^{\ka,r}$ and norm $\|\cdot\|_{\hat{W}^{\ka,r}}$ used below may be found in \cref{ssec:introNot}. 

\begin{thm}\label{thm:main}
Let $d\geq 1$, $\ga>0$, $\max(\frac{1}{2},\frac{2-\ga}{2})<s\leq 1$. Suppose that $\g$ satisfies \eqref{eq:gassmp} for $\ga$. Given $\al,\be,\nu>0$, set $\phi^t\coloneqq \al+\beta t$ and
\begin{align}\label{eq:zetadef}
\zeta \coloneqq \inf_{k\in \Z^d : k\neq 0}\left(\frac{\nu^2}{2}-\beta|k|^{-s}  - \Re\left(\hat{\g}(k)|k|^{-2s}(k\cdot \M  k)\right)\right).
\end{align}
Assume that $\zeta>0$.

If $s$ is sufficiently large depending on $\ga$, then there exists an $r_0\geq 1$ depending on $d,\ga,s$, such that the following holds. For any $1\leq r \leq r_0$ and any $\sigma>0$ sufficiently large depending on $d,\ga,r,s$, there is a constant $C>0$ depending only on $d,\ga,r,s,\sigma$, such that for initial data $\mu^0$ satisfying $\frac{1}{(2\pi)^d}\int_{\T^d}\mu^0dx=1$ and
\begin{equation}\label{eq:idcon}
\left\|e^{(\al+\ep)\Dm^{s}}\mu^0-1\right\|_{\hat{W}^{\sigma s,r}} < \frac{\zeta}{C|\M|},
\end{equation}
for $\ep>0$, and any path in $\Omega_{\al,\be,\nu}$, there exists a unique global solution $\mu\in C^0([0,\infty); \hat{W}^{\sigma s,r})$ to equation \eqref{eq:rpde} with initial datum $\mu^0$. Moreover,
\begin{align}
\forall t\geq 0, \qquad   \left\|e^{(\phi^t+\ep)\Dm^s}\mu^t-1\right\|_{\hat{W}^{\sigma s,r}} \leq e^{-\frac{\zeta t}{2}} \left\|e^{(\al+\ep)\Dm^s}\mu^0-1\right\|_{\hat{W}^{\sigma s,r}}.
\end{align}
\end{thm}

\begin{remark}\label{rem:size}
To make the statement of \cref{thm:main} reader-friendly, we have opted not to include the explicit relations between parameters, such as $d,\ga,s,r_0,\sigma$. These relations are explicitly worked out in \Cref{sec:LWP,sec:Equi}. Throughout the paper, the reader should keep in mind that the most favorable choice is $(s,r) = (1,1)$.
\end{remark}

\begin{remark}\label{rem:zeta}
The condition $s>\frac{2-\ga}{2}$ ensures that we can make $\zeta>0$ by fixing $\beta,\g,\M$ and then taking $\nu$ sufficiently large. 
\end{remark}

\begin{remark}\label{rem:rescale}
By rescaling time and using conservation of mass (see \cref{rem:mcon} below), we may always reduce to the case $\frac{1}{(2\pi)^d}\int_{\T^d}\mu^0 dx=1$ up to a change of $\nu$. More precisely, suppose that $\mu$ is a solution to \eqref{eq:rpde}. Letting $m = \frac{1}{(2\pi)^d}\int_{\T^d}\mu^0 dx$, set $\mu_m^t \coloneqq \frac{1}{m}\mu^{t/m}$. Then using the chain rule,
\begin{align}
\p_t\mu_m^t &= -\frac{1}{m^2}\Gamma^{\frac{t}{m}}\div\paren*{\left(\Gamma^{\frac{t}{m}}\right)^{-1}\mu^{t/m}\M\nabla\g\ast\left(\Gamma^{\frac{t}{m}}\right)^{-1}\mu^{t/m} } -\frac{\nu^2}{2m^2}\Dm^{2s}\mu^{t/m} \nn\\
&= -\Gamma_m^t\div\paren*{\left(\Gamma_m^t\right)^{-1}\mu_m^t\nabla\g\ast\left(\Gamma_m^t\right)^{-1}\mu_m^t} - \frac{\nu_m^2}{2} \Dm^{2s}\mu_m^t,
\end{align}
where $\nu_m \coloneqq \nu/\sqrt{m}$, $W_m^t \coloneqq \sqrt{m}W^{t/m}$, and $\Gamma_m^t \coloneqq e^{-\nu_m W_m^t \Dm^s}$. Note that $W_m$ is again a standard Brownian motion (e.g., see \cite[Lemma 9.4]{KS1991}).
\end{remark}

As advertised at the beginning of the introduction, our main result shows that with quantifiable high probability, solutions of the random PDE \eqref{eq:rpde} with Gevrey initial data are global and as $t\rightarrow\infty$, converge exponentially fast in Gevrey norm to the uniform distribution with the same mass as $\mu^t$. The essential point and importance of our work is that our result is agnostic to $\M$ (no gradient flow or repulsive-type assumptions) and to $\g$, subject to the very general condition \eqref{eq:gassmp}. This generality means our result covers equations for which global existence, let alone asymptotic behavior, is unknown or for which finite-time blow-up happens in the deterministic case.

The long-time behavior of equation \eqref{eq:spde} with $\nu=0$ is highly dependent on the nature of $\M$ and the singularity of $\g$. In general, little is known in the Hamiltonian case where $\M$ is antisymmetric. For instance, if $d=2$, $\M$ is rotation by $\frac{\pi}{2}$, and $\hat{\g}(k)=|k|^{-2}$, the equation becomes the incompressible Euler vorticity equation (see \cite[Section 1.2] {MP2012book}, \cite[Chapter 2]{MB2002}). Global well-posedness of classical/weak solutions \cite{Wolibner1933, Holder1933, Yudovich1963}  is known, but the asymptotic behavior is only partially understood (e.g., see \cite{Shnirelman2013, KS2014, BM2015id, MZ2020, IJ2020, IJ2022, DD2022, DE2022} and references therein). For the same choice of $d,\M$, if  $\hat{\g}(k)=|k|^{-\ga}$, for $\ga \in (0,2)$, then equation \eqref{eq:spde} becomes the inviscid generalized SQG equation \cite{CMT1994, PHS1994spec, HPGS1995, CCCGW2012}. Global existence of smooth solutions to the gSQG equation is a major open problem \cite{CF2002, Gancedo2008, CGI2019, CCG2020, BCCK2020, GP2021gsqg, HK2021, CCZ2021}. It is only known if one adds suitably deterministic strong diffusion (e.g., see \cite{CW1999QG, KNV2007gwp, CV2010, CV2012nmp}). In the gradient case where $\M=\mp \I$, global existence vs. finite-time blow-up depends on the choice of sign. We discuss only the model interaction $\hat{\g}(k) = |k|^{-\ga}$ which is sometimes called a fractional porous medium equation. Local well-posedness of classical solutions is known \cite{CJ2021}. But in the attractive case $\I$, suitably strong solutions blow up in finite time \cite{BLL2012}. In the repulsive case $-\I$, global existence, uniqueness, and asymptotic behavior of nonnegative classical and $L^\infty$ weak solutions are known when $\ga=2$ \cite{LZ2000, AS2008, BLL2012, SV2014} (see also \cite{MZ2005, AMS2011, Mainini2012}). The easier case $\ga>2$ follows by the same arguments \cite[Section 4]{CCH2014} (see also \cite{BLR2011}). For $0<\ga<2$, global existence, regularity, and asymptotic behavior of certain nonnegative weak solutions are known \cite{CV2011, CV2011asy, CSV2013, CV2015, BIK2015, CHSV2015, LMS2018}; but per our knowledge, these weak solutions are only known to be unique if $d=1$ \cite{BKM2010}. It is an open problem whether classical solutions are global if $0<\ga<2$. In the interests of completeness, we also mention there is a large body of work on the long-time behavior of the gradient case for regular potentials $\g$ satisfying convexity assumptions (on $\R^d$). For example, see \cite{Villani2004}, to which the title of our paper pays homage.\footnote{Many of the references discussed in this paragraph are set on $\R^d$; but in general, these results have analogues on $\T^d$.} 

There is an extensive literature on the effects of noise (e.g., ``regularization by noise''), a sample of which is contained in the references \cite{dBD2002, dBD2005, FGP2010, FGP2011, DT2011, Flandoli2011, GhV2014, CG2015, BFGM2019, GG2019, BNSW2020, FL2021, FGL2021, MST2021, BMX2023}. But to our knowledge, these previous works have not investigated the equilibriating properties of stochastic perturbations. Related in spirit to our work, we mention some works \cite{FM1995, Mattingly1999, EMS2001, BKL2001, HM2006, CGhV2014, FFGhR2017} on the ergodicity of fluid equations subject to stochastic forcing. But we emphasize these results add noise to a diffusive deterministic model, for which a result comparable to ours is already known (e.g., see \cite{GW2005} for 2D Navier-Stokes), and are instead about the balance between the injection of energy through noise and the dissipation of energy through viscosity. 

\subsection{Comments on the proof}\label{ssec:imtroCP}
The proof of \cref{thm:main} builds on the previous work \cite{RS2023}. The key point there to obtaining global solutions is a monotonicity formula for the Gevrey norm, asserting that it is strictly decreasing, provided the initial data and parameters are appropriately chosen. Showing this monotonicity requires carefully estimating the size of nonlinearity and showing it does not overwhelm the dissipative effect of the diffusion. In the present work, we go a step further by considering the evolution equation satisfied the unknown $\varrho^t \coloneqq \mu^t-1$. We show a dissipation inequality for the Gevrey norm of $\varrho^t$, which, under suitable conditions on the initial data, allows us to deduce the exponential-in-time decay of the Gevrey norm of $\varrho^t$ through a delicate continuity argument and Gr\"onwall's lemma.

One might ask why we work on the torus for the equation \eqref{eq:spde}, as opposed to on $\R^d$ for the equation
\begin{align}\label{eq:inhompde}
\p_t\theta +\div(\theta\M\nabla\g\ast\theta)  = \nu(1+\Dm^s)\theta\dot{W}^t
\end{align}
originally considered in \cite{RS2023}. The periodic setting is technically simpler since the spectrum is discrete and one does not have to worry about low-frequency issues, in particular, when $\ga>1$. This allows to replace the inhomogeneous multiplier $(1+\Dm^s)$, which kills off all Fourier modes, by $\Dm^s$, which kills off only nonzero Fourier modes. We expect that a similar analysis can be performed for \eqref{eq:inhompde} on $\R^d$ \emph{mutatis mutandis}, where now with high probability, $\mu^t \coloneqq e^{-\nu W^t(1+\Dm^s)}\theta^t$ should converge to zero (vacuum) in Gevrey norm as $t\rightarrow\infty$. 

Finally, let us mention that our method is quite robust and would also work, for example, for the periodic 3D incompressible Euler equation modified by random diffusion (alternatively, the 3D Navier-Stokes with white noise modulated hyperviscosity):
\begin{align}\label{eq:NSE}
\p_t u + u \cdot\nabla u = -\nabla p + \nu\Dm^s u\dot{W}.
\end{align}
This becomes clearer from rewriting \eqref{eq:NSE} in Leray projector form. One can show that with high probability, the transformed unknown $v^t \coloneqq \Gamma^t u^t$ converges in Gevrey norm exponentially fast as $t\rightarrow\infty$ to the vector $\int_{\T^3}v^0 dx$. To minimize the length of the paper, we leave such extensions to the interested reader.

\subsection{Organization of paper}\label{ssec:introorg}
We briefly comment on the organization of the remaining body of the paper. \cref{sec:LWP} introduces the scale of Gevrey function spaces, some elementary embeddings for these spaces, and then uses these spaces to show the local well-posedness for equation \eqref{eq:rpde}, with the main result being \cref{prop:lwp}. \cref{sec:Equi} then shows the global existence and exponential decay to equilibrium, completing the proof of \cref{thm:main}. This is spread over two preliminary results: \cref{prop:mon} and \cref{lem:lspan}.

\subsection{Notation}\label{ssec:introNot}
Let us conclude the introduction by reviewing the essential notation of the paper, following the conventions of \cite{RS2023}.

Given nonnegative quantities $A$ and $B$, we write $A\lesssim B$ if there exists a constant $C>0$, independent of $A$ and $B$, such that $A\leq CB$. If $A \lesssim B$ and $B\lesssim A$, we write $A\sim B$. To emphasize the dependence of the constant $C$ on some parameter $p$, we sometimes write $A\lesssim_p B$ or $A\sim_p B$. 

The Fourier and inverse transform of a function $f:\T^d\rightarrow\C^m$ are given by
\begin{equation}
\begin{split}
\hat{f}(k) = \F(f)(k) &\coloneqq \int_{\T^d}f(x)e^{-ix\cdot k}dx,\\
\check{f}(x) = \F^{-1}(f)(x) &\coloneqq \frac{1}{(2\pi)^d}\sum_{k\in\Z^d}f(k)e^{ik\cdot x},
\end{split}
\end{equation}
The homogeneous Bessel potential space $\dot{W}^{s,p}$ is defined by
\begin{equation}
\|f\|_{\dot{W}^{s,p}} \coloneqq \||\nabla|^{s}f\|_{L^p}, \qquad s\in\R, \ p\in (1,\infty),
\end{equation}
and the Fourier-Lebesgue space $\hat{W}^{s,p}$ is defined by
\begin{equation}\label{defn:fl_space}
\|f\|_{\hat{W}^{s,p}} \coloneqq \||\cdot|^s\hat{f}\|_{\ell^p}, \qquad s\in\R, \ p\in [1,\infty].
\end{equation}
$C^0([0,T); X)$ denotes the space of functions taking values in the Banach space $X$, which are continuous and bounded.

\section{Local well-posedness}\label{sec:LWP}
We show local well-posedness for the equation \eqref{eq:rpde}, the main result of this section being \cref{prop:lwp} stated below. This proposition---and its proof via a contraction mapping argument---is a modification of \cite[Proposition 3.1]{RS2023}. Although it was noted in \cite[Remark 1.7]{RS2023} that the results from that paper have corresponding analogues on the torus, we present the proof anyway because it is not written anywhere else and, more importantly, the two-tier function space (see (3.7) in the cited work) used on $\R^d$ becomes unnecessary on the torus because the spectrum is discrete. We also improve on \cite[Proposition 3.1]{RS2023} (and the earlier \cite[Proposition 3.1]{BNSW2020}) by removing the smallness condition $\beta < \frac{\nu^2}{2}$, which explains why the statement may not look comparable.

Set $A\coloneqq \Dm^{2s}$ and define the bilinear operator
\begin{equation}\label{eq:Bdef}
B(f,g) \coloneqq \div\Gamma\paren*{\Gamma^{-1}f(\M\nabla\g\ast\Gamma^{-1}g)}.
\end{equation}
Strictly speaking, $B$ is time-dependent through $\Gamma$. When necessary, we make explicit this time dependence by writing $B^t(f,g)$. Assume $\frac{1}{(2\pi)^d}\int_{\T^d}\mu^0dx=1$. In contrast to \cite{RS2023}, it will be more convenient to work with the unknown $\varrho^t \coloneqq \mu^t - 1$, which satisfies the equation
\begin{align}
\p_t\varrho^t &=-\div\Gamma\paren*{\Gamma^{-1}\varrho^t\M\nabla\g\ast\Gamma^{-1}\varrho^t } - \div(\M\nabla\g\ast\varrho^t) - \frac{\nu^2}{2}\Dm^{2s}\varrho^t \nn\\
&= -B^t(\varrho^t, \varrho^t) -L \varrho^t -\frac{\nu^2}{2}A\varrho^t, \label{eq:varrhoeqn}
\end{align}
where $L \coloneqq \div\left(\M\nabla\g\ast(\cdot)\right)$. Note that $L=0$ if $\M$ is antisymmetric. If we have a solution $\varrho^t$ to \eqref{eq:varrhoeqn}, then $\mu^t\coloneqq 1+\varrho^t$ is a solution to \eqref{eq:rpde}. So, there is no loss in working with the unknown $\varrho^t$. We rewrite the Cauchy problem for \eqref{eq:varrhoeqn} in the mild form,
\begin{equation}\label{eq:mild}
\varrho^t = e^{-t\left(\frac{\nu^2 A}{2}+L\right)}\varrho^0 - \int_0^t e^{-(t-\tau)\left(\frac{\nu^2 A}{2}+L\right)} B^\tau(\varrho^\tau,\varrho^\tau)d\tau.
\end{equation}

\begin{remark}\label{rem:symbLB}
Observe that the real part of the symbol of $\frac{\nu^2 A}{2}+L$ is
\begin{align}
\frac{\nu^2|k|^{2s}}{2} + \Re(\hat{\g}(k))\left(\M k\cdot k\right) \geq \frac{\nu^2|k|^{2s}}{2} - C|\M| |k|^{2-\ga},
\end{align}
where we have used assumption \eqref{eq:gassmp} to obtain the lower bound. If $2s\geq 2-\ga$, then for all $|k|\geq \left(\frac{2C|\M|}{\nu^2}\right)^{\frac{1}{2s-2+\ga}}$,  the symbol of $\frac{\nu^2 A}{2}+L$ has nonnegative real part.
\end{remark}

To perform a contraction mapping argument based on \eqref{eq:mild}, we use a scale of Gevrey function spaces from \cite{RS2023} (see \cite{FT1989, BNSW2020} for earlier $L^2$ special cases). For $a \geq 0$, $\kappa \in \mathbb{R}$, define
\begin{align}\label{defn:gevery norm}
\|f\|_{\G_a^{\kappa, r}} \coloneqq \left\|e^{a A^{1 / 2}} f\right\|_{\hat{W}^{\kappa s, r}}.
\end{align}
For $0<T<\infty$ and a continuous function $\phi:[0,T]\rightarrow [0,\infty)$, we define
\begin{equation}
\|f\|_{C_T^0{\G}_{\phi}^{\ka,r}} \coloneqq \sup_{0\leq t\leq T} \|f^t\|_{\G_{\phi^t}^{\ka,r}}.
\end{equation}
We write $C_{\infty}^0$ when $\sup_{0\leq t\leq T}$ is replaced by $\sup_{0\leq t<\infty}$. Define the Banach space
\begin{equation}
C_T^0\G_{\phi}^{\ka,r} \coloneqq \{f\in C([0,T]; \hat{W}^{\ka s,r}(\T^d)) : \|f\|_{C_T^0\G_{\phi}^{\ka,r}} < \infty\}.
\end{equation}
We also allow for $T=\infty$, replacing $[0,T]$ in the preceding line with $[0,\infty)$.

\begin{prop}\label{prop:lwp}
Let $d\geq 1$, $\ga>0$, $\max(\frac{1}{2},\frac{2-\ga}{2})<s\leq 1$. Given $\al,\be,\nu>0$, suppose $W$ is a realization from the set $\Omega_{\al,\be,\nu}$ and set $\phi^t\coloneqq \al+\beta t$.

There exists $r_0 \geq 1$ depending on $d,\ga,s$, such that the following holds. For any $1\leq r\leq r_0$, there exists $\sigma_0 \in (0,\frac{2s-1}{s})$ depending on $d,\ga,r,s$, such that for any $\sigma \in (\sigma_0,\frac{2s-1}{s})$ with $1-\ga\leq \sigma s$, there exists a time $T>0$ such that for  $\|\varrho^0\|_{\G_{\al}^{\sigma,r}} \leq R$, there exists a unique solution $\varrho \in C_T^0 \G_{\phi}^{\sigma,r}$ to the Cauchy problem for \eqref{eq:varrhoeqn}. Moreover,
\begin{equation}
\|\varrho\|_{C_T^0\G_{\al}^{\sigma,r}} \leq 2\|\varrho^0\|_{\G_{\al}^{\sigma,r}}.
\end{equation}
Additionally, if $\|\varrho_j^0\|_{\G_{\al}^{\sigma,r}}\leq R$, for $j\in\{1,2\}$, then
\begin{equation}
\|\varrho_1-\varrho_2\|_{C_T^0\G_{\al}^{\sigma,r}} \leq 2\|\varrho_1^0-\varrho_2^0\|_{\G_{\al}^{\sigma,r}}.
\end{equation}
\end{prop}


\begin{remark}\label{rem:mcon}
The solutions given by \cref{prop:lwp} conserve mass, hence solutions to the original equation \eqref{eq:rpde} also conserve mass. One readily sees this by integrating both sides of \eqref{eq:mild} over $\T^d$ and by using the fundamental theorem of calculus together with $\mathcal{F}\left(e^{-t\left(\frac{\nu^2 A}{2}+mL\right)}\right)(0) = 1$. Thus,
\begin{equation}
\int_{\T^d}\varrho^t(x)dx =\int_{\T^d}\varrho^0(x)dx = 0.
\end{equation}
\end{remark}

\subsection{Gevrey and Sobolev embeddings}\label{ssec:lwpGev}
Before proceeding to the proof of \cref{prop:lwp}, we record some elementary embeddings satisfied by the spaces $\G_a^{\ka,r}$. For proofs of the following lemmas, see \cite[Section 2.2]{RS2023}. 

\begin{lemma}\label{lem:Gemb}
If $a'\geq a\geq 0$ and $\ka'\geq \ka$, then
\begin{align}
\|f\|_{\G_{a}^{\ka,r}} \leq e^{a-a'}\|f\|_{\G_{a'}^{\ka',r}}.
\end{align}
If $\ka'\geq \ka$ and $a'>a\geq 0$, then
\begin{equation}
\|f\|_{\G_{a}^{\ka',r}} \leq  \frac{\lceil{\ka'-\ka}\rceil !}{(a'-a)^{\lceil{\ka'-\ka}\rceil}} \|f\|_{\G_{a'}^{\ka,r}},
\end{equation}
where $\lceil{\cdot}\rceil$ denotes the usual ceiling function.
\end{lemma}

\begin{lemma}\label{lem:Sob}
If $1\leq p <r\leq\infty$, then
\begin{equation}
\|f\|_{\hat{W}^{s,p}} \lesssim_{d,p,r} \|f\|_{\hat{W}^{(s+\frac{d(r-p)}{rp})+,r}},
\end{equation}
where the notation $(\cdot)+$ means $(\cdot)+\vep$, for any $\vep>0$, with the implicit constant then depending on $\vep$ and possibly blowing up as $\vep\rightarrow 0^+$. If $2\leq p\leq \infty$, then if $\hat{f}(0) =0$,
\begin{equation}
\|f\|_{\hat{W}^{s,p}} \lesssim_{d,p} \|f\|_{\dot{W}^{s,\frac{p}{p-1}}}.
\end{equation}
\end{lemma}

\subsection{Contraction mapping argument}\label{ssec:lwpcm}
Throughout this subsection, assume that we have fixed a realization of $W$ from $\Omega_{\al,\be,\nu}$. Fix $\varrho^0$ and define the map
\begin{equation}
\varrho^t\mapsto (\Tc\varrho)^t \coloneqq e^{-t\left(\frac{\nu^2 A}{2}+L\right)}\varrho^0 - \int_0^t e^{-(t-\tau)\left(\frac{\nu^2 A}{2}+L\right)} B^\tau(\varrho^\tau,\varrho^\tau)d\tau.
\end{equation}
We check that $\Tc$ is well-defined on $C_T^0\G_{\phi}^{\sigma,r}$ for $\phi^t=\al+\be t$, with $\al,\be,\sigma,r>0$ satisfying the conditions in the statement of \cref{prop:lwp}. 

First, we control the linear term in \eqref{eq:mild}. We introduce some notation that will be used in what follows. Define the parameters
\begin{align}
|k_0| \coloneqq \sup\left\{|k|: k \in \Z^d, \ \beta |k|^s + \Re\left(\hat{\g}(k)\M k\cdot k\right)-\frac{\nu^2 |k|^{2s}}{2} \geq 0 \right\}, \label{eq:kadef} \\
\la \coloneqq \sup_{k\in\Z^d} \left(\beta |k|^s + \Re\left(\hat{\g}(k)\M k\cdot k\right)-\frac{\nu^2 |k|^{2s}}{2} \right).\label{eq:ladef}
\end{align}
Since $2s>\max(2-\ga,s)$ by assumption, $|k_0|$ is finite and
\begin{align}
\la = \sup_{k:|k|\leq |k_0|}  \left(\beta |k|^s + \Re\left(\hat{\g}(k)\M k\cdot k\right)-\frac{\nu^2 |k|^{2s}}{2} \right).
\end{align}

\begin{lemma}\label{lem:cmlin}
For any $1\leq r\leq\infty$, $\max(\frac{2-\ga}{2},0)<s\leq 1$, $\sigma\in\R$, and $\al,\be,\nu>0$, it holds that 
\begin{equation}
\forall t\geq 0,\qquad \|e^{-t\left(\frac{\nu^2 A}{2}+L\right)}f\|_{\G_{\phi^t}^{\sigma,r}} \leq e^{\la t }\|f\|_{\G_{\al}^{\sigma,r}}.
\end{equation}
\end{lemma}
\begin{proof}
Unpacking the definition of the $\G_{\phi^t}^{\sigma, r}$ norm, it holds that
\begin{align}
\|e^{-t\left(\frac{\nu^2 A}{2}+L\right)}f\|_{\G_{\phi^t}^{\sigma,r}}^r &= \left\|e^{\phi^t A^{1 / 2}} e^{-t\left(\frac{\nu^2 A}{2}+L\right)}f\right\|_{\hat{W}^{\sigma s,r}}^r \nn\\
&=\sum_{k}|k|^{r s \sigma}\left|e^{\phi^t|k|^s-\frac{\nu^2|k|^{2s}}{2} + \left(\M k\cdot k\right)\hat{\g}(k)} \hat{f}(k)\right|^r \nn\\
&= \left[\sum_{|k|\leq |k_0|} + \sum_{|k| >|k_0|}\right]|k|^{r s \sigma}e^{r\alpha |k|^s}\left|e^{ t\left(\beta |k|^s + \Re\left(\hat{\g}(k)\M k\cdot k\right)-\frac{\nu^2 |k|^{2s}}{2}\right)} \hat{f}(k)\right|^r \nn\\
&\leq \sum_{|k|\leq |k_0|} |k|^{r s \sigma}e^{r\alpha |k|^s}e^{rt\la }|\hat{f}(k)|^r +\sum_{|k|> |k_0|}|k|^{r s \sigma}e^{r\alpha |k|^s}  |\hat{f}(k)|^{r}\nn\\
&\leq e^{rt \la }\|e^{\al A^{1/2}}f\|_{\hat{W}^{\sigma s,r}}^r,
\end{align}
where the final line follows from $\la \geq 0$.
\end{proof}

Next, we control the bilinear term in \eqref{eq:mild}.
\begin{lemma}\label{lem:cmnl}
Let $d\geq 1$, $\ga>0$, $\max(\frac{2-\ga}{2},\frac12)<s\leq 1$. There exists an $r_0 \in [1,\infty]$, depending on $d,s$, such that the following holds. For any $1\leq r \leq r_0$, there exists $\sigma_0 \in (0,\frac{2s-1}{s})$ depending on $d,s,r$, such that for any $\sigma \in (\sigma_0, \frac{2s-1}{s})$ with $1-\ga\leq \sigma s$, there exists a constant $C$ depending only on $d,r,q,\sigma,s,\beta,\nu$, such that for any $T>0$,
\begin{multline}
\left\|\int_0^t e^{-\left(\frac{\nu^2A}{2} + L\right)}B^\tau(\varrho_1^\tau,\varrho_2^\tau)d\tau\right\|_{C_T^0\G_{\phi}^{\sigma,r}} \\
\leq C|\M| \|\varrho_1\|_{C_T^0\G_{\phi}^{\sigma,r}} \|\varrho_2\|_{C_T^0\G_{\phi}^{\sigma,r}} \Bigg(|k_0|^{s\sigma}\left(\frac{e^{t\la}-1}{\la} \right) + T^{1-\frac{\sigma s+1}{2s}}\Bigg).
\end{multline}
\end{lemma}
\begin{proof}[Proof of \cref{lem:cmnl}]
We make the change of unknown $\varrho_j^t\coloneqq e^{-\phi^t A^{1/2}}|\nabla|^{-\sigma s} \rho_j^t$, so that
\begin{equation}
\|\rho_j^t\|_{\hat{L}^r} = \|\varrho_j^t\|_{\G_{\phi^t}^{\sigma,r}}.
\end{equation}
By Minkowski's inequality, we see that
\begin{align}
\left\|e^{\phi^t A^{1/2}} \int_0^{t}e^{-(t-\tau)\left(\frac{\nu^2 A}{2}+L\right)}B^\tau(\varrho_1^\tau,\varrho_2^\tau) d\tau\right\|_{\hat{W}^{\sigma s,r}} \leq \int_0^t \left\|e^{\phi^t A^{1/2} - (t-\tau)\left(\frac{\nu^2 A}{2}+L\right)}B^\tau(\varrho_1^\tau,\varrho_2^\tau)\right\|_{\hat{W}^{\sigma s,r}}d\tau,
\end{align}
and by definition of the $\hat{W}^{\sigma s,r}$ norm, the preceding right-hand side equals
\begin{multline}
\int_0^t \Bigg(\sum_{k}e^{r(t-\tau)\left(\beta |k|^s + \Re(\hat{\g}(k)\M k\cdot k)-\frac{\nu^2 |k|^{2s}}{2}\right)} |k|^{r\sigma s}\\
 \left|\sum_{j} \frac{|k\cdot\M j| |\hat{\g}(j)|}{|k-j|^{\sigma s}|j|^{\sigma s}} e^{\left(\phi^\tau-\nu W^\tau\right)\left[|k|^s-|k-j|^s-|j|^s\right]} \hat{\rho}^\tau_1(k-j) \hat{\rho}^\tau_2(j)\right|^r \Bigg)^{1/r} d \tau.
\end{multline}
We adopt the notational convention $\frac{\hat{\rho}_1^\tau(k-j)}{|k-j|^{\sigma s}} \coloneqq 0$ when $k=j$ (similarly, for $\hat{\rho}_2^\tau$). Using $\phi^t-\phi^\tau = \be(t-\tau)$, the preceding expression is controlled by
\begin{multline}
\int_0^t \Bigg(\sum_{k}e^{r(t-\tau)\left(\beta |k|^s + \Re(\hat{\g}(k)\M k\cdot k)-\frac{\nu^2 |k|^{2s}}{2}\right)}|k|^{r\sigma s}\\
\left|\sum_{j } \frac{|k\cdot\M j| |\hat{\g}(j)|}{|k-j|^{\sigma s}|j|^{\sigma s}} e^{\left(\phi^\tau-\nu W^\tau\right)\left[|k|^s-|k-j|^s-|j|^s\right]} \hat{\rho}^\tau_1(k-j) \hat{\rho}^\tau_2(j)\right|^r \Bigg)^{1/r} d \tau. \label{eq:insumpre}
\end{multline}
Since $0<s\leq 1$, we have $|k|^s - |k-j|^s-|j|^s\leq 0$ for all $k,j\in \Z^d$. Since $\phi^\tau-\nu W^\tau\geq 0$ for all $0\leq\tau\leq t$ by assumption, it follows that
\begin{equation}
e^{\left(\phi^\tau-\nu W^\tau\right)\left[|k|^s-|k-j|^s-|j|^s\right]} \leq 1.
\end{equation} 
Thus, for fixed $k$, estimating the inner sum of \eqref{eq:insumpre}, we find
\begin{align}
&\left|\sum_{j} \frac{|k\cdot\M j| |\hat{\g}(j)|}{|k-j|^{\sigma s}|j|^{\sigma s}} e^{\left(\phi^\tau-\nu W^\tau\right)\left[|k|^s-|k-j|^s-|j|^s\right]} \hat{\rho}^\tau_1(k-j) \hat{\rho}^\tau_2(j)\right|^r  \nn\\
&\lesssim \left(\sum_{j} \frac{|k\cdot\M j| |\hat{\g}(j)|}{|k-j|^{\sigma s}|j|^{\sigma s}}  |\hat{\rho}^\tau_1(k-j)| |\hat{\rho}^\tau_2(j)|\right)^r \nn\\
&\lesssim  |\mathbb{M}|^r |k|^r \left(\sum_{j } |k-j|^{-\sigma s}  |\hat{\rho}^\tau_1(k-j)| |j|^{1-\gamma-\sigma s} |\hat{\rho}^\tau_2(j)|\right)^r,
\end{align}
where we have implicitly used that $\g$ satisfies \eqref{eq:gassmp} to obtain the last line. With $|k_0|$ defined as in \eqref{eq:kadef} above, there exists a constant $\d>0$, such that for frequencies $|k|>|k_0|$, 
\begin{equation}
\frac{\nu^2 |k|^{2s}}{2} - \beta |k|^s - \Re(\hat{\g}(k)\M k\cdot k)\geq \d |k|^{2s}.
\end{equation}
Furthermore, observe that by writing $|k| = (t-\tau)^{-\frac{1}{2s}}(t-\tau)^{\frac{1}{2s}}|k|$, it follows from the power series for $z\mapsto e^{z}$ that
\begin{equation}
e^{r(t-\tau)\left(\beta |k|^s + \Re(\hat{\g}(k)\M k\cdot k)-\frac{\nu^2 |k|^{2s}}{2}\right) }|k|^{rs\sigma} \leq e^{-r\d(t-\tau)|k|^{2s}}|k|^{r\sigma s} \lesssim_\d (t-\tau)^{-\frac{r(\sigma s+1)}{2s}}.
\end{equation}
For frequencies $|k|\leq |k_0|$ (of which there are at most finitely many), we crudely estimate
\begin{align}
e^{r(t-\tau)\left(\beta |k|^s + \Re(\hat{\g}(k)\M k\cdot k)-\frac{\nu^2 |k|^{2s}}{2}\right) }|k|^{rs\sigma} \leq e^{r(t-\tau)\la}|k_0|^{rs\sigma},
\end{align}
with $\la$ as in \eqref{eq:ladef}. With these observations, we find
\begin{multline}
\int_0^t \Bigg(\sum_{k}e^{r(t-\tau)\left(\beta |k|^s + \Re(\hat{\g}(k)\M k\cdot k)-\frac{\nu^2 |k|^{2s}}{2}\right)}|k|^{r\sigma s} \\
\left|\sum_{j} \frac{|k\cdot\M j| |\hat{\g}(j)|}{|k-j|^{\sigma s}|j|^{\sigma s}} e^{\left(\phi^\tau-\nu W^\tau\right)\left[|k|^s-|k-j|^s-|j|^s\right]} \hat{\rho}^\tau_1(k-j) \hat{\rho}^\tau_2(j)\right|^r \Bigg)^{1/r} d \tau\\
\lesssim |\M||k_0|^{s\sigma}\int_0^t e^{(t-\tau)\la} \left(\sum_{|k|\leq |k_0|}\left(\sum_{j}|k-j|^{-\sigma s}  |\hat{\rho}^\tau_1(k-j)| |j|^{1-\gamma-\sigma s} |\hat{\rho}^\tau_2(j)|\right)^r \right)^{1/r} d \tau\\
+ |\M| \int_0^t (t-\tau)^{-\frac{(\sigma s+1)}{2s}} \left(\sum_{|k|>|k_0|}\left(\sum_{j}|k-j|^{-\sigma s}  |\hat{\rho}^\tau_1(k-j)| |j|^{1-\gamma-\sigma s} |\hat{\rho}^\tau_2(j)|\right)^r \right)^{1/r} d \tau.
\end{multline}
Thus, it remains to estimate the factor containing the $\ell_k^r$ norm of the sum over $j$. For this, we use Young's inequality followed by Sobolev embedding \cref{lem:Sob},
\begin{align}
&\left(\sum_{k}\left(\sum_{j}|k-j|^{-\sigma s}  |\hat{\rho}^\tau_1(k-j)| |j|^{1-\gamma-\sigma s} |\hat{\rho}^\tau_2(j)|\right)^r \right)^{1/r} \nn\\
&\leq \||\cdot|^{-\sigma s}\hat{\rho}_1^\tau\|_{\ell^p} \||\cdot|^{1-\ga-\sigma s} \hat{\rho}_2^\tau\|_{\ell^{\frac{rp}{(r+1)p-r}}} \nn\\
&\lesssim  \|\rho_1^\tau\|_{\hat{W}^{-\sigma s, 1}} \|\rho_2^\tau\|_{\hat{W}^{1-\ga-\sigma s, 1}}\indic_{r=1} + \|\rho_1^\tau\|_{\hat{W}^{(\frac{d(r-1)}{r}-\sigma s)+, r}} \|\rho_2^\tau\|_{\hat{W}^{1-\ga-\sigma s, r}}\indic_{\substack{p=1 \\ r>1}} \nn\\
&\ph + \|\rho_1^\tau\|_{\hat{W}^{-\sigma s,r}} \|\rho_2^\tau\|_{\hat{W}^{(1-\ga-\sigma s + \frac{d(r-1)}{r})+, r}}\indic_{\substack{ p=r \\r>1}}\nn\\
&\ph +\|\rho_1^\tau\|_{\hat{W}^{(\frac{d(r-p)}{rp}-\sigma s)+,r}} \|\rho_2^\tau\|_{\hat{W}^{(1-\ga-\sigma s+\frac{d(p-1)}{p})+, r}}\indic_{\substack{1<p<r \\ r>1}} \nn\\
&= \|e^{\phi^\tau A^{1/2}}\varrho_1^\tau\|_{\hat{W}^{0,1}}\|e^{\phi^\tau A^{1/2}}\varrho_2^\tau\|_{\hat{W}^{1-\ga,1}}\indic_{r=1} + \|e^{\phi^\tau A^{1/2}}\varrho_1^\tau\|_{\hat{W}^{\frac{d(r-1)}{r}+,r}}\|e^{\phi^\tau A^{1/2}}\varrho_2^\tau\|_{\hat{W}^{1-\ga,r}}\indic_{\substack{p=1 \\ r>1}} \nn\\
&\ph +  \|e^{\phi^\tau A^{1/2}}\varrho_1^\tau\|_{\hat{W}^{0,r}}\|e^{ \phi^\tau A^{1/2}}\varrho_2^\tau\|_{\hat{W}^{(1-\ga + \frac{d(r-1)}{r})+,r}}\indic_{\substack{ p=r \\r>1}}\nn \\
&\ph + \|e^{\phi^\tau A^{1/2}}\varrho_1^\tau\|_{\hat{W}^{(\frac{d(r-p)}{rp})+,r}} \| e^{\phi^\tau A^{1/2}}\varrho_2^\tau\|_{\hat{W}^{(1-\ga+\frac{d(p-1)}{p})+, r}}\indic_{\substack{1<p<r \\ r>1}}, \label{eq:lwphf}
\end{align}
where the final equality follows from unpacking the definition of $\rho^t$. To obtain estimates that close, the top Sobolev index appearing in \eqref{eq:lwphf} must be $\leq \sigma s$. This leads to the following conditions:
\begin{equation}
\begin{cases}
1-\ga \leq \sigma s, & {r=1} \\
\frac{d(r-1)}{r}<\sigma s \ \text{and} \  1-\ga \leq \sigma s, & {p=1 \ \text{and} \ r>1} \\
1-\ga+\frac{d(r-1)}{r}<\sigma s, & {p=r \ \text{and} \ r>1} \\
\frac{d(r-p)}{rp}<\sigma s \ \text{and} \ 1-\ga+\frac{d(p-1)}{p} < \sigma s, & {1<p<r \ \text{and} \ r>1}.
\end{cases}
\end{equation}
Assuming the preceding conditions are met and also that $\frac{(\sigma s+1)}{2s} <1$, it follows from our work that
\begin{multline}\label{eqn:bi-ineq}
\left\|e^{\phi^t A^{1/2}} \int_0^{t}e^{-(t-\tau)\left(\frac{\nu^2 A}{2}+L\right)}B^\tau(\varrho_1^\tau,\varrho_2^\tau) d\tau\right\|_{\hat{W}^{\sigma s,r}} \lesssim |\M||k_0|^{s\sigma}\left(\frac{e^{T\la}-1}{\la} \right)\| \varrho_1\|_{C_T^0\G_{\phi}^{\sigma,r}} \|\varrho_2\|_{C_T^0G_{\phi}^{\sigma,r}} \\
+|\M| T^{1-\frac{\sigma s+1}{2s}}\| \varrho_1\|_{C_T^0\G_{\phi}^{\sigma,r}} \|\varrho_2\|_{C_T^0G_{\phi}^{\sigma,r}}
\end{multline}
for any $0\leq t\leq T$. We adopt the convention that the first term in the preceding right-hand side is zero if $\la=0$.

To complete the proof of the lemma, it is important to list all the conditions we imposed on the parameters $d,\ga,\sigma,s,r$ during the course of the above analysis:
\begin{enumerate}[(LWP1)]
\item\label{sCon}
$\frac{2-\ga}{2}<s\leq 1$;
\item\label{pCon}
	\begin{enumerate}
	\item\label{pCona}
	$r=1$ and $1-\ga\leq \sigma s$,
	\item\label{pConb}
	or $r>1$ and $\frac{d(r-1)}{r}<\sigma s$ and $1-\ga\leq \sigma s$,
	\item\label{pConc}
	or $r>1$ and $1-\ga + \frac{d(r-1)}{r}<\sigma s$,
	\item\label{pCond}
	or $r>1$ and $\exists p\in (1,r)$ such that $\frac{d(r-p)}{rp}<\sigma s$ and $1-\ga +\frac{d(p-1)}{p}<\sigma s$;
	\end{enumerate}
\item\label{ssCon}
$\frac{(\sigma s+1)}{2s}<1$
\end{enumerate}
We refer the reader to the proof of \cite[Lemma 3.7]{RS2023} for the existence of a non-trivial choice of parameters satisfying the above conditions. 
\end{proof}

\medskip

\begin{proof}[Proof of \cref{prop:lwp}]
Putting together the estimates of \Cref{lem:cmlin,lem:cmnl}, we have shown that there exists a constant $C>0$ depending on $d,\ga,r,\sigma,s,\beta,\nu$, such that
\begin{equation}\label{eq:Tmu}
\|\Tc(\varrho)\|_{C_T^0\G_{\phi}^{\sigma,r}} \leq e^{T\la}\|\varrho^0\|_{\G_{\phi}^{\sigma,r}} + C|\M| \|\varrho\|_{C_T^0\G_{\phi}^{\sigma,r}}^2  \Bigg(|k_0|^{s\sigma}\left(\frac{e^{T\la}-1}{\la} \right) + T^{1-\frac{\sigma s+1}{2s}}\Bigg)
\end{equation}
and
\begin{multline}\label{eq:Tmu12}
\|\Tc(\varrho_1)-\Tc(\varrho_2)\|_{C_T^0\G_{\phi}^{\sigma,r}} \leq C|\M|\Bigg(|k_0|^{s\sigma}\left(\frac{e^{T\la}-1}{\la} \right) + T^{1-\frac{\sigma s+1}{2s}}\Bigg)\|\varrho_1-\varrho_2\|_{C_T^0\G_{\phi}^{\sigma,r}}\\
\times \paren*{\|\varrho_1\|_{C_T^0\G_{\phi}^{\sigma,r}}+\|\varrho_2\|_{C_T^0\G_{\phi}^{\sigma,r}}}.
\end{multline}
We want to show that for any appropriate choice of $T$, the map $\Tc$ is a contraction on the closed ball $B_{R}(0)$ of radius $R\geq 2\|\varrho^0\|_{\G_{\phi}^{\sigma,r}}$ centered at the origin in the space $C_T^0\G_{\phi}^{\sigma,r}$. From the estimates \eqref{eq:Tmu} and \eqref{eq:Tmu12}, we see that if
\begin{align}
&e^{T\la} \leq \frac{3}{2}, \label{eq:lwpT1}\\
&2C|\M|R\Bigg(|k_0|^{s\sigma}\left(\frac{e^{T\la}-1}{\la} \right) + T^{1-\frac{\sigma s+1}{2s}}\Bigg) {\leq} \frac{1}{8}, \label{eq:lwpT2}
\end{align}
then $\Tc$ is a contraction on $B_R(0)$. So by the contraction mapping theorem, there exists a unique fixed point $\varrho=\Tc(\varrho) \in C_T^0\G_{\phi}^{\sigma,r}$. We let $T_0$ denote the maximal $T$ such that \eqref{eq:lwpT1}, \eqref{eq:lwpT2} both hold. We note that the maximal lifespan of the solution is $\geq T_0$. 

The preceding result shows the local existence and uniqueness. To complete the proof of \cref{prop:lwp}, we now prove continuous dependence on the initial data. For $j=1,2$, let $\varrho_j$ be a solution in $C_{T_j}^{0}\G_{\phi}^{\sigma,r}$ to \eqref{eq:varrhoeqn} with initial datum $\varrho_j^0$, such that $\|\varrho_j^0\|_{\G_{\phi}^{\sigma,r}}\leq R$. From the mild formulation \eqref{eq:mild}, the triangle inequality, \Cref{lem:cmlin,lem:cmnl}, we see that
\begin{multline}\label{eq:cdre}
\|\varrho_1-\varrho_2\|_{C_T^0\G_{\phi}^{\sigma,r}} \leq \|\varrho_1^0-\varrho_2^0\|_{\G_{\phi}^{\sigma,r}} \\
+ C|\M|\Bigg(|k_0|^{s\sigma}\left(\frac{e^{T\la}-1}{\la} \right) + T^{1-\frac{\sigma s+1}{2s}}\Bigg)\|\varrho_1-\varrho_2\|_{C_T^0\G_{\phi}^{\sigma,r}}\paren*{\|\varrho_1\|_{C_T^0\G_{\phi}^{\sigma,r}} + \|\varrho_2\|_{C_T^0\G_{\phi}^{\sigma,r}}}.
\end{multline}
Taking $T$ smaller if necessary while still preserving $T\gtrsim T_0$, we may assume that
\begin{align}
2C|\M|\Bigg(|k_0|^{s\sigma}\left(\frac{e^{T\la}-1}{\la} \right) + T^{1-\frac{\sigma s+1}{2s}}\Bigg) \leq\frac{1}{4}.
\end{align}
Bounding each $\|\varrho_j\|_{C_T^0\G_{\phi}^{\sigma,r}}$ by $R$ in the last factor, it then follows from \eqref{eq:cdre} that
\begin{equation}
\|\varrho_1-\varrho_2\|_{C_T^0\G_{\phi}^{\sigma,r}}\leq 2\|\varrho_1^0-\varrho_2^0\|_{\G_{\phi}^{\sigma,r}}.
\end{equation}
This last estimate completes the proof of \cref{prop:lwp}.
\end{proof}

\begin{remark}\label{rem:TextLB}
An examination of the proof of \cref{prop:lwp} reveals that when $\la\leq 0$, which is implied by the assumption $\zeta\geq 0$ (recall \eqref{eq:zetadef}), the time of existence $T$ given by the fixed point argument satisfies the lower bound
\begin{align}
T \geq C(|\M|R)^{-\frac{2s}{(2-\sigma)s-1}},
\end{align}
where the constant $C>0$ depends quantitatively on the parameters $d,\ga,s,\sigma,r,\beta,\nu$.
\end{remark}

\section{Global existence and convergence to equilibrium}\label{sec:Equi}
We now conclude the proof of \cref{thm:main} by showing that the solutions are global and converge to the uniform distribution as $t\rightarrow\infty$.

Assume that $\frac{1}{(2\pi)^d}\int_{\T^d}\mu^0 dx=1$. Let $\varrho^t = \mu^t -1$ be as in \cref{sec:LWP}, and recall that $\varrho^t$ satisfies the equation
\begin{align}
\p_t\varrho^t &= -B^t(\varrho^t, \varrho^t) -L \varrho^t -\frac{\nu^2}{2}A\varrho^t,
\end{align}
where we remind the reader that $A=\Dm^{2s}$, the operator $B$ was defined in \cref{eq:Bdef}, and $L= \div\left(\M\nabla\g\ast(\cdot)\right)$. Also, recall that $\hat{\varrho}^0(0) = \int_{\T^d}\varrho^0dx =0$, so by conservation of mass (\cref{rem:mcon}), $\hat{\varrho}^t(0)=0$ for every $t\geq 0$.

Our first result shows that if $\varrho$ belongs to a higher regularity Gevrey space on $[0,T]$, then the norm associated to a lower regularity Gevrey space decays exponentially in time on $[0,T]$. 

\begin{prop}\label{prop:mon}
Let $d\geq 1$, $\ga>0$, $1\leq r\leq \infty$, $\max(\frac{1}{2},\frac{2-\ga}{2})<s\leq 1$. Given $\al,\be,\nu>0$, set $\phi^t\coloneqq \al+\be t$ and assume that $W$ is a realization from $\Omega_{\al,\be,\nu}$. Define 
\begin{align}\label{eq:zetadef'}
\zeta \coloneqq \inf_{k\in \Z^d : k\neq 0}\left(\frac{\nu^2}{2}-\beta|k|^{-s} +\Re(\hat{\g}(k))|k|^{-2s}\M k\cdot  k\right)
\end{align}
and assume that $\zeta>0$.

There is a threshold $\ka_{0}\in\R$ depending on $r,d,s,\ga$, such that for any $\ka>\ka_{0}$, the following holds. There is a constant $C>0$, depending only on $d,\ga,r,s,\ka$, such that if $\varrho \in C_T^0\G_{\phi}^{\ka+\frac{2}{r},r}$ is a solution to \eqref{eq:varrhoeqn}, for some $T>0$, satisfying
\begin{equation}\label{eq:monidr}
\|\varrho^0\|_{\G_{\al}^{\ka,r}} < \frac{\zeta}{C|\M|},
\end{equation}
then
\begin{equation}\label{eq:monr}
\forall t\in [0,T], \qquad \|\varrho^t\|_{\G_{\phi^t}^{\ka,r}} \leq  e^{-\frac{\zeta t}{2}}\|\varrho^0\|_{\G_{\al}^{\ka,r}}  .
\end{equation}
\end{prop}

\medskip
The starting point of the proof of \cref{prop:mon} (cf. \cite[Section 4.1]{RS2023}) is to compute for $k\in\Z^d$,
\begin{multline}\label{eq:dtabspre}
\frac{d}{d t}\left|e^{\phi^t |k|^s} \hat{\varrho}^t(k)\right|=\Re\left(|e^{\phi^t | k |^s}\hat{\varrho}^t(k)|^{-1} \overline{ e^{\phi^t | k |^s}\hat{\varrho}^t(k)}\left(\beta|k|^s e^{\phi^t|k|^s} \hat{\varrho}^t(k)\right.\right. \\
\left.\quad-e^{\phi^t|k|^s} \mathcal{F}\left(B\left(\varrho^t, \varrho^t\right)\right)(k) + \left(k\cdot \M k \right)\hat{\g}(k)e^{\phi^t|k|^s}\hat{\varrho}^t(k)-\frac{\nu^2}{2}|k|^{2s} e^{\phi^t|k|^s} \hat{\varrho}^t(k)\right).
\end{multline}
Majorizing the nonlinear term by its absolute value, we obtain
\begin{align}\label{eq:ddtvrhok}
    \frac{d}{d t}\left|e^{\phi^t|k|^s} \hat{\varrho}^t(k)\right|  &\leq-\left|e^{\phi^t|k|^s} \hat{\varrho}^t(k)\right|\left(\frac{\nu^2}{2}|k|^{2 s}-\beta|k|^s - \Re(\hat{\g}(k))\M k\cdot k\right)+\left|e^{\phi^t|k|^s} \mathcal{F}\left(B^t\left(\varrho^t, \varrho^t\right)\right)(k)\right|.
\end{align}
Using \eqref{eq:ddtvrhok}, we compute
\begin{align}
\frac{1}{r} \frac{d}{d t}\left\|e^{\phi^t A^{1 / 2}} \varrho^t\right\|_{\hat{W}^{\sigma s, r}}^r &=\frac{1}{r} \frac{d}{dt} \sum_{k} |k|^{r\sigma s} \left|e^{\phi^t |k|^s} \hat{\varrho}^t(k)\right|^r \nn\\
&=\sum_{k} |k|^{r\sigma s} \left|e^{\phi^t |k|^s} \hat{\varrho}^t(k)\right|^{r-1} \frac{d}{d t}\left|e^{\phi^t|k|^s} \hat{\varrho}^t(k)\right| \nn\\
&\leq-\zeta \sum_{k} \left|e^{\phi^t|k|^s}|k|^{\left(\sigma+\frac{2}{r}\right) s} \hat{\varrho}^t(k)\right|^r  \nn\\
&\quad+\sum_{k} \left|e^{\phi^t|k|^s}|k|^{\sigma s} \hat{\varrho}^t(k)\right|^{r-1}|k|^{\sigma s} e^{\phi^t |k|^s}\left|\mathcal{F}\left(B^t\left(\varrho^t, \varrho^t\right)\right)(k)\right|.\label{eq:ddtWmonpre}
\end{align}

To estimate the nonlinear term in the preceding right-hand side, we use two lemmas, which are periodic analogues of \cite[Lemmas 4.3, 4.4]{RS2023}, respectively. We omit their proofs, as the arguments are essentially the same as the $\ga\leq 1$ Euclidean case.

\begin{lemma}\label{lem:Bbndpre}
For any $t\geq 0$ with $\phi^t - \nu W^t\geq 0$, it holds for any test functions $f,g$ that
\begin{equation}
|e^{\phi^t|k|^s}\F(B^t(f,g))(k)| \lesssim_\ga |\M|\sum_{j\neq 0}|k| |j|^{1-\ga} \left|e^{\phi^t|k-j|^s}\hat{f}(k-j) e^{\phi^t|j|^s}\hat{g}(j)\right|.
\end{equation}
\end{lemma}

\begin{lemma}\label{lem:Bbnd}
Let $d\geq 1$, $\ga>0$, $1\leq r\leq \infty$, $\frac{1}{2}< s\leq 1$. Then there exists a threshold $\ka_0$ depending on $d,\ga,r,s$, such that for any $\ka>\ka_0$, there exists a constant $C>0$ depending on $d,\ga,r,s,\ka$ so that
\begin{multline}\label{eq:Bbnd}
\sum_{k} \left|e^{\phi^t|k|^s}|k|^{\ka s} \hat{h}(k) \right|^{r-1}|k|^{\ka s+1}\sum_{j\neq 0}|j|^{1-\ga} \left|e^{\phi^t|k-j|^s}\hat{f}(k-j) e^{\phi^t|j|^s}\hat{g}(j)\right|\\
\leq C\|e^{\phi^t A^{1/2}} h\|_{\hat{W}^{(\ka+\frac{2}{r})s,r}}^{r-1}\Bigg(\|e^{\phi^t A^{1/2}}f\|_{\hat{W}^{(\ka + \frac{2}{r})s,r}} \|e^{\phi^t A^{1/2}}g\|_{\hat{W}^{\ka s,r}} + \|e^{\phi^t A^{1/2}}f\|_{\hat{W}^{\ka s,r}}\|e^{\phi^t A^{1/2}}g\|_{\hat{W}^{(\ka + \frac{2}{r})s,r}}\Bigg).
\end{multline}
\end{lemma}

\medskip
\begin{proof}[Conclusion of proof of \cref{prop:mon}]
Applying \Cref{lem:Bbndpre,lem:Bbnd} with $f=g=h=\varrho^t$ and $\ka>\ka_0$, and choosing $\sigma=\ka$ in the inequality \eqref{eq:ddtWmonpre}, we find that
\begin{align}\label{eq:wtsnn}
\frac{d}{d t} \frac{1}{r}\left\|e^{\phi^t A^{1 / 2}} \varrho^t\right\|_{\hat{W}^{\kappa s,r}}^{r} &\leq-\zeta\left\|e^{\phi^t A^{1 / 2}} \varrho^t\right\|_{\hat{W}^{(\kappa+\frac{2}{r}) s, r}}^r  + C|\mathbb{M}| \|e^{\phi^tA^{1/2}} \varrho^t\|^{r}_{\hat{W}^{(\kappa+\frac{2}{r})s,r}} \|e^{\phi^t A^{1 / 2}} \varrho^t\|_{\hat{W}^{\kappa s, r}} \nonumber\\
&= \left\|e^{\phi^t A^{1 / 2}} \varrho^t\right\|_{\hat{W}^{\left(\kappa+\frac{2}{r}\right) s, r}}^{r}\left(C|\mathbb{M}|\left\|e^{\phi^t A^{1 / 2}} \varrho^t\right\|_{\hat{W}^{\kappa s, r}}-\zeta\right).
\end{align}
If we assume that 
\begin{align}
    \left\|e^{\al A^{1 / 2}} \varrho^0\right\|_{\hat{W}^{\kappa s, r}} < \frac{\zeta}{2C|\mathbb{M}|},
\end{align}
where $C$ is the same constant as in \eqref{eq:wtsnn}, then we claim that this inequality persists for all time $t \in [0,T]$. We argue by contradiction. Let $T_*\geq 0$ denote the maximal time such that
\begin{equation}\label{eq:ubhold}
\forall t\in [0,T_*), \qquad \|e^{\phi^t A^{1/2}}\varrho^t\|_{\hat{W}^{\ka s, r}} < \frac{\zeta}{2C|\mathbb{M}|}.
\end{equation}
Such a $T_*$ exists and is positive since the preceding inequality is true at $t=0$ by assumption and the function $t\mapsto \|e^{\phi^t A^{1/2}}\varrho^t\|_{\hat{W}^{\ka s,r}}$ is continuous. If $T_*=T$, then there is nothing to prove, so assume otherwise. \eqref{eq:ubhold} together with \eqref{eq:wtsnn} imply that $t\mapsto \|e^{\phi^t A^{1/2}}\varrho^t\|_{\hat{W}^{\ka s, r}}$ is strictly decreasing on $[0,T_*)$ (assuming $\varrho^t$ is a nonzero solution), implying
\begin{equation}
\|e^{\phi^t A^{1/2}}\varrho^{T_*}\|_{\hat{W}^{\ka s,r}}  <\|e^{\phi^t A^{1/2}}\varrho^0\|_{\hat{W}^{\ka s, r}} < \frac{\zeta}{2C|\mathbb{M}|}.
\end{equation}
This inequality implies by maximality that $T_*=T$. Therefore, for $t\in [0,T]$,
\begin{align}
   \frac{d}{d t}\left\|e^{\phi^t A^{1 / 2}} \varrho^t\right\|_{\hat{W}^{\kappa s, r}}^{r}  \leq -\frac{r\zeta}{2}\left\|e^{\phi^t A^{1 / 2}} \varrho^t\right\|_{\hat{W}^{\left(\kappa+\frac{2}{r}\right) s, r}}^{r}  \leq -\frac{r\zeta}{2} \left\|e^{\phi^t A^{1 / 2}} \varrho^t\right\|_{\hat{W}^{\kappa s, r}}^{r}.
\end{align}
Applying Gr\"onwall's lemma, we conclude that 
\begin{align}
\forall t\in [0,T], \qquad    \left\|e^{\phi^t A^{1 / 2}} \varrho^t\right\|_{\hat{W}^{\kappa s, r}}^{r} \leq e^{-\frac{r\zeta t}{2}} \|e^{\alpha A^{1 / 2}} \varrho^0\|_{\hat{W}^{\kappa s, r}}^r,
\end{align}
which completes the proof of \cref{prop:mon}.
\end{proof}

On its own, \cref{prop:mon} does not imply \cref{thm:main} because the former assumes that $\varrho^t$ lives in a higher index Gevrey space on $[0,T]$, while only showing that a lower index Gevrey norm of $\varrho^t$ decays on $[0,T]$. The lower index norm does not control the higher index norm, so somehow we have to make up for this discrepancy between spaces. 


Fix $\ep>0$ and suppose that $\varrho^0\in \G_{\al+\ep}^{\sigma_0,r}$ for $\sigma_0$ above the regularity threshold $\ka_{0}$ given by \cref{prop:mon}. Assume that the parameters $d,\ga,r,s,\sigma_0,\al,\be,\nu$ satisfy all the constraints of \cref{thm:main} and also assume that
\begin{equation}\label{eq:sigma0tz}
\|\varrho^0\|_{\G_{\al+\ep}^{\sigma_0,r}} < \frac{\zeta}{C_{exp}|\M|},
\end{equation}
where $\zeta$ is as in \eqref{eq:zetadef'} and $C_{exp}>0$ is the constant from \cref{prop:mon}. Assuming a realization of $W$ from $\Omega_{\al,\be,\nu}$ and given $r\geq 1$ sufficiently small depending on $d,\ga,s$, \cref{prop:lwp} implies that for any $0<\sigma<\frac{2s-1}{s}$, with $1-\ga\leq \sigma s$, sufficiently large depending on $d,\ga,s,r$, there is a maximal solution $\varrho$ to equation \eqref{eq:varrhoeqn} with lifespan $[0,T_{\max,\sigma,\ep})$, such that $\varrho$ belongs to $C_T^0\G_{\phi+\ep}^{\sigma,r}$ for any $0\leq T<T_{\max,\sigma,\ep}$. The main lemma to conclude global existence relates the lifespan of $\varrho^t$ in $\G_{\phi^t+\ep}^{\sigma,r}$ to the lifespan of $\varrho^t$ in the \emph{larger space} $\G_{\phi^t+\ep'}^{\sigma,r}$, for any $\ep'\in [0,\ep)$. For details on how to prove such a result, see the proof of \cite[Lemma 4.5]{RS2023}, bearing in mind \cref{rem:TextLB}.

\begin{lemma}\label{lem:lspan}
Let $\varrho$ be as above. There exists a constant $C>0$ depending on $d,\ga,r,s,\sigma,\be,\nu$ such that for any $0\leq \ep_2<\ep_1\leq \ep$, the maximal times of existence $T_{\max,\sigma,\ep_1}, T_{\max,\sigma,\ep_2}$ of $\varrho^t$ as taking values in $\G_{\phi^t+\ep_1}^{\sigma,r}, \G_{\phi^t+\ep_2}^{\sigma,r}$, respectively, satisfy the inequality
\begin{equation}\label{eq:Tmaxeplb}
T_{\max,\sigma,\ep_2} \geq T_{\max,\sigma,\ep_1} + C(|\M|\|\varrho^0\|_{\G_{\al+\ep}^{\sigma_0,r}})^{-\frac{2s}{2s-\sigma s-1}}.
\end{equation}
\end{lemma}

\medskip
\begin{proof}[Proof of \cref{thm:main}]
Fix $0<\ep'<\ep$, and let $\sigma,\sigma_0$ be as above. If $T_{\max, \sigma, \ep'}<\infty$, then let $n\in\N$ be such that $n C (|\M|\|\varrho^0\|_{\G_{\al+\ep}^{\sigma_0,r}})^{-\frac{2s}{2s-\sigma s-1}}$ satisfies the inequality
\begin{equation}
n C(|\M|\|\varrho^0\|_{\G_{\al+\ep}^{\sigma_0,r}})^{-\frac{2s}{2s-\sigma s-1}} > T_{\max,\sigma,\ep'}-T_{\max,\sigma,\ep},
\end{equation}
where $C$ is the same constant as in the inequality \eqref{eq:Tmaxeplb}. We observe from \cref{lem:lspan} that
\begin{align}
T_{\max,\sigma,\ep'} - T_{\max,\sigma,\ep} &= \sum_{j=0}^{n-1} \paren*{T_{\max, \sigma,\ep - \frac{(j+1)(\ep-\ep')}{n}} - T_{\max,\sigma,\ep-\frac{j(\ep-\ep')}{n}}} \nn\\
&\geq \sum_{j=0}^{n-1} C(|\M|\|\varrho^0\|_{\G_{\al+\ep}^{\sigma_0,r}})^{-\frac{2s}{2s-\sigma s-1}}\nn\\
&> T_{\max,\sigma,\ep'}-T_{\max,\sigma,\ep},
\end{align}
which is a contradiction. Thus, $T_{\max,\sigma,\ep'}=\infty$.

For any $0<\ep'<\ep$ and any $0<\sigma<\frac{2s-1}{s}$ sufficiently large depending on $d,\ga,s,r$, it therefore holds that $\|\varrho\|_{C_T^0\G_{\phi+\ep'}^{\sigma,r}}<\infty$ for all $T>0$. Using the arbitrariness of $\ep'$, \cref{lem:Gemb} implies that for any $T>0$, $\|\varrho\|_{C_T^0\G_{\phi+\ep'}^{\sigma_0+\frac{2}{r},r}} < \infty$. Using that
\begin{equation}
\|\varrho^0\|_{\G_{\al+\ep'}^{\sigma_0,r}} \leq \|\varrho^0\|_{\G_{\al+\ep}^{\sigma_0,r}} < \frac{\zeta}{C_{exp}|\M|}
\end{equation}
by assumption \eqref{eq:sigma0tz}, we can apply \cref{prop:mon} on the interval $[0,T]$ to obtain that
\begin{align}\label{eq:exp0T}
\forall t\in [0,T], \qquad \|\varrho^t\|_{\G_{\phi^t+\ep'}^{\sigma_0,r}} \leq  e^{-\frac{\zeta t}{2}}\|\varrho^0\|_{\G_{\al+\ep'}^{\sigma_0,r}} .
\end{align}
Since $T>0$ was arbitrary, the decay \eqref{eq:exp0T}, in fact, holds on $[0,\infty)$.

Finally, we can replace $\ep'$ in both sides of \eqref{eq:exp0T} by the larger $\ep$. Indeed, the result of the preceding paragraph and the trivial inequality $\|\cdot\|_{\G_{\al+\ep'}^{\sigma_0,r}} \leq \|\cdot\|_{\G_{\al+\ep}^{\sigma_0,r}}$, for $\ep'\leq\ep$, give
\begin{equation}
\forall t\geq 0, \qquad \|\varrho^t\|_{\G_{\phi^t+\ep'}^{\sigma_0,r}} \leq  e^{-\frac{\zeta t}{2}}\|\varrho^0\|_{\G_{\al+\ep'}^{\sigma_0,r}} \leq e^{-\frac{\zeta t}{2}}\|\varrho^0\|_{\G_{\al+\ep}^{\sigma_0,r}} < \infty.
\end{equation}
The desired conclusion now follows by unpacking the definition of the left-hand side and appealing to the monotone convergence theorem.

\end{proof}

\bibliographystyle{alpha}
\bibliography{../MASTER}
\end{document}